\documentclass[a4paper,11pt]{amsart}
\usepackage{amsmath,amsthm,amssymb}
\usepackage[mathscr]{eucal}
 \usepackage{cite}
\usepackage{upgreek}
\usepackage[bookmarks,bookmarksnumbered]{hyperref}
\usepackage{enumerate}

\numberwithin{equation}{section}

\theoremstyle{plain}
\newtheorem{main}{Theorem}

\newtheorem{theorem}{Theorem}[section]

\newtheorem{lemma}[theorem]{Lemma}
\newtheorem{proposition}[theorem]{Proposition}
\newtheorem{corollary}[theorem]{Corollary}
\theoremstyle{definition}
\newtheorem{definition}[theorem]{Definition}
\newtheorem{example}[theorem]{Example}

\newtheorem{remark}[theorem]{Remark}

\begin{document}

\title[Cocycle superrigidity for the coinduced action]
{Cocycle superrigidity for coindunced actions}
\author[Daniel Drimbe]{Daniel Drimbe}

\address{Mathematics Department; University of California, San Diego, CA 90095-1555 (United States).}
\email{ddrimbe@ucsd.edu}

\begin{abstract} 

We prove a cocycle superrigidity theorem for a large class of coinduced actions. In particular, if $\Lambda$ is a subgroup of a countable group $\Gamma$, we consider a probability measure preserving action $\Lambda\curvearrowright X_0$ and let $\Gamma\curvearrowright X$ be the coinduced action. Assume either that $\Gamma$ has property (T) or that $\Lambda$ is amenable and $\Gamma$ is a product of non-amenable groups. Using Popa's deformation/rigidity theory we prove $\Gamma\curvearrowright X$ is $\mathcal U_{fin}$-cocycle superrigid, that is any cocycle for this action to a $\mathcal U_{fin}$ (e.g. countable) group $\mathcal V$ is cohomologous to a homomorphism from $\Gamma$ to $\mathcal V.$
\end{abstract}

\maketitle

\section{Introduction and statement of main results}

\subsection{Introduction}

The goal of this article is to prove a general cocycle superrigidity theorem for {\it coinduced actions} (see Definition \ref{defca}) and derive several consequences to orbit equivalence and von Neumann algebras. 

The classification of probability measure preserving (pmp) actions of countable groups on standard probability spaces up to orbit equivalence has attracted a lot of interest in the last 15 years (see the surveys \cite{Po07, Fu09, Ga10, Va10a, Io12a}). Two pmp actions $\Gamma\curvearrowright (X,\mu)$ and $\Lambda\curvearrowright (Y,\nu)$ are {\it orbit equivalent} (OE) if there exists a measure space isomorphism $f:X\to Y$ which sends orbits to orbits, i.e. $f(\Gamma x)=\Lambda f(x),$ for almost every $x\in X.$ 



If the groups are amenable, the classification up to orbit equivalence  is done. More precisely, Orstein and Weiss proved in \cite{OW80} (see also \cite{Dy58, CFW 81}) that all free ergodic pmp actions of countable amenable groups are orbit equivalent. In contrast, the non-amenable case is much more challenging and complex. Remarkably, several classes which are {\it rigid} in the sense that one can deduce conjugacy from OE, have been discovered. The most extreme form of rigidity  for orbit equivalence is OE-superigidity: $\Gamma\curvearrowright X$ is OE-{\it superrigid} if every free ergodic pmp action which is OE with $\Gamma\curvearrowright X$ is conjugate with it. The first OE-superrigidity result was obtained by Furman in the late 1990s by building on Zimmer's cocycle superrigidity \cite{Zi84}. He showed that many actions of higher rank lattices, including the action $SL_n(\mathbb Z)\curvearrowright \mathbb T^n,$ for $n\ge 3$ is OE-superrigid \cite{Fu98,Fu99}. After this, a number of striking OE-superrigidity results were obtained \cite{MS02,Po05,Po06,Ki06,Io08,PV08,Ki09,PS09,Io14,TD14,CK15}.

In particular, in his breakthrough work \cite{Po05,Po06}, Popa used his deformation/rigidity theory to prove a remarkable cocycle superrigidity theorem for Bernoulli actions of groups with property (T) and of products of non-amenable groups. More precisely, if $\Gamma\curvearrowright X$ is such an action, Popa obtained that every cocycle with values in a countable (and more generally, in a $\mathcal U_{fin}$) group is cohomologous with a group homomorphism.
By applying his cocycle superrigidity theorem to cocycles arising from orbit equivalence, he proved that the action $\Gamma\curvearrowright X$ is OE-superrigid.


\subsection{Statement of the main results}\label{results}

Our main result provides a generalization of Popa's cocycle superrigidity theorem to coinduced actions. We first review some basic concepts starting with the construction of coinduced actions (see e.g. \cite{Io06b}).

\begin{definition}\label{defca}
Let $\Gamma$ be a countable group and let $\Lambda$ be a subgroup. Let $\phi:\Gamma/\Lambda \to \Gamma$ be a section. Define the cocycle $c:\Gamma\times\Gamma/\Lambda\to \Lambda$ by the formula $c(g,x)=\phi^{-1}(gx)g\phi(x),$ for all $g\in\Gamma$ and $x\in \Gamma/\Lambda.$\\
Let $\Lambda\overset{\sigma_{0}}{\curvearrowright} X_0$ be a pmp action, where $(X_0,\mu_0)$ is a standard probability space. We define an action $\Gamma\overset{\sigma}{\curvearrowright} X_0^{\Gamma/\Lambda}$, called the coinduced action of $\sigma_0$, as follows:
$$\sigma_{g}((x_{h})_{h\in \Gamma/\Lambda})=(x'_{h})_{h\in \Gamma/\Lambda},$$
where $(x_{h})_{h\in \Gamma/\Lambda}\in X_0^{\Gamma/\Lambda}$ and $x'_{h}=c(g,h)x_{g^{-1}h}$.
Note that $\sigma$ is a pmp action of $\Gamma$ on the standard probability space $X_0^{\Gamma/\Lambda}$, where $X_0^{\Gamma/\Lambda}$ is endowed with the product measure $\mu_0^{\Gamma/\Lambda}.$
\end{definition}

\begin{remark}
If we consider the trivial action of $\Lambda=\{e\}$ on $X_0$, then the coinduced action of $\Gamma$ on $X_0^{\Gamma/\{e\}}=X_{0}^{\Gamma}$ is the Bernoulli action.
\end{remark}

We say that the inclusion $\Gamma_0\subset\Gamma$ of countable groups has the {\it Kazhdan's relative property} (T) if for every $\epsilon>0$, there exist $\delta>0$ and $F\subset\Gamma$ finite such that if $\pi:\Gamma\to \mathcal U (K)$ is a unitary representation and $\xi\in K$ is a unit vector satisfying $\|\pi(g)\xi-\xi\|<\delta$, for all $g\in F$, then there exists $\xi_0\in K$ such that $\|\xi-\xi_0\|<\epsilon$ and $\pi(h)\xi_0=\xi_0$, for all $h\in \Gamma_0.$  The group $\Gamma$ has the {\it property} (T) if the inclusion $\Gamma\subset \Gamma$ has the relative property (T). To give some example, $\mathbb Z^2\subset \mathbb Z^2\rtimes SL_2(\mathbb Z)$ has the relative property (T) and $SL_n(\mathbb Z),$ $n\ge 3$, has the property (T) \cite{Ka67,Ma82}.

An infinite subgroup $H$ of $\Gamma$ is {\it w-normal} in $\Gamma$ if there exist an ordinal $\beta$ and intermediate subgroups $H=H_{0}\subset H_{1}\subset\dots\subset H_{\beta}=\Gamma$ such that for all $0<\alpha\leq \beta$, the group $\cup_{\alpha'<\alpha}H_{\alpha'}$ is normal in $H_{\alpha}.$ Denote by $\mathcal U_{fin}$ the class of Polish groups which arise as closed subgroups of the unitary groups of II$_1$ factors. In particular, all countable discrete groups and all compact Polish groups belong to $\mathcal U_{fin}.$ These two notions are due to Popa \cite{Po05}.

For a Polish group $G$, a measurable map $w:\Gamma\times X\to G$ is called a $cocycle$ if it satisfies the relation $w(\gamma_1\gamma_2,x)=w(\gamma_1,\gamma_2 x)w(\gamma_2,x)$, for all $\gamma_1,\gamma_2 \in \Gamma$ and for almost every $x\in X.$ Two cocycles $w,w':\Gamma\times X\to G$ are $cohomologous$ if there exists a measurable map $\phi:X\to G$ such that $w'(\gamma,x)=\phi(\gamma x)w(\gamma,x)\phi(x)^{-1}$, for all $\gamma\in \Gamma$ and for almost every $x\in X$. An action $\Gamma\curvearrowright(X,\mu)$ is called $\mathcal U_{fin}${\it -cocycle superrigid} if every cocycle with values in a group from $\mathcal U_{fin}$ is cohomologous with a group homomorphism.

The following theorem is our first main result, which generalizes Popa's cocycle superrigidity theorem for Bernoulli actions of property (T) groups to coinduced actions (see \cite{Po05} and also \cite{Fu06, Va06}).

\begin{main}[Groups with relative property (T)]\label{A}
Let $\Gamma$ be a countable group and $\Lambda$ be a subgroup. Let $H\subset\Gamma$ be a subgroup with relative property (T). Assume that there does not exist a finite index subgroup $H_0$ of $H$ which is contained in a conjugate $g^{-1}\Lambda g$ of $\Lambda,$ for some $g\in \Gamma.$ 

Take $\mathcal{V}\in \mathcal{U}_{fin}$. Let $\sigma_{0}$ be a pmp action of $\Lambda$ on a standard probability space $(X_{0},\mu_0)$  and $\sigma$ the coinduced action of $\Gamma$ on $X:=X_{0}^{\Gamma/\Lambda}$.

Then, any cocycle $w:\Gamma\times X\to \mathcal{V}$ for the restriction of $\sigma$ to $H$ is cohomologous to a group homomorphism  $d:H\to \mathcal{V}$.\\
Moreover, if $H$ is w-normal in $\Gamma$, then $w$ is cohomologous to a group homomorphism $d:\Gamma\to\mathcal{V}$ and therefore $\Gamma\curvearrowright X$ is  U$_{fin}${\it -cocycle superrigid}.
\end{main}

In particular, Theorem \ref{A} implies that if $\Gamma$ has property (T) (e.g. $\Gamma=SL_n(\mathbb Z), n\ge 3$) and $\Lambda$ is an infinite index subgroup of $\Gamma$ (e.g. $\Lambda$ is cyclic), then any coinduced action of $\Gamma$ from $\Lambda$ is $\mathcal U_{fin}$-cocycle superrigid.

In \cite[Corollary 1.2]{Po06}, Popa proved a cocycle superrigidity theorem for the Bernoulli action of product groups analogous with \cite[Corollary 5.4]{Po05}. 
The next theorem generalizes this result to coinduced actions.

\begin{main}[Product groups]\label{B}
Let $\Gamma$ be a countable group and $\Lambda$ be an amenable subgroup. Let $H$ and $H'$ be infinite commuting subgroups of $\Gamma$ such that $H'$ is non-amenable. Assume that there does not exist a finite index subgroup $H_0$ of $H$ which is contained in a conjugate $g^{-1}\Lambda g$ of $\Lambda,$ for some $g\in \Gamma.$ 

Take $\mathcal{V}\in \mathcal{U}_{fin}$. Let $\sigma_{0}$ be a pmp action of $\Lambda$ on a standard probability space $(X_{0},\mu_0)$ and $\sigma$ the coinduced action of $\Gamma$ on $X:=X_{0}^{\Gamma/\Lambda}$.

Then, any cocycle $w:\Gamma\times X\to \mathcal{V}$ for the restriction of $\sigma$ to $HH'$ is cohomologous to a group homomorphism  $d:HH'\to \mathcal{V}$.\\
Moreover, if $H$ is w-normal in $\Gamma$, then $w$ is cohomologous to a group homomorphism $d:\Gamma\to\mathcal{V}$ and therefore $\Gamma\curvearrowright X$ is  $\mathcal U_{fin}${\it -cocycle superrigid}.
\end{main}

The proof of Theorem \ref{B} goes along the same lines as the proof of \cite[Theorem 4.1]{Po06}. First, we untwist the cocycle on $H$ using the rigidity gained from the non-amenability of $H'$ (instead of using property (T) as in Theorem \ref{A}). Then, using weak mixing properties of coinduced actions and the fact that $H$ is normal in $HH'$, we are able to untwist the cocycle on $HH'$.

We will prove in this paper a more general version of Theorems \ref{A} and \ref{B} dealing with coinduced actions of $\Gamma$ on $A^{\Gamma/\Lambda}$ that arise from actions of $\Lambda$ on arbitrary tracial von Neumann algebras $A$. 

As an immediate consequence of Theorems \ref{A} and \ref{B}, we deduce the following  OE-superrigidity result for coinduced actions. 

\begin{corollary}[OE-superrigidity]\label{OEsuper}
Let $\Gamma$ be a countable subgroup with no non-trivial finite normal subgroups and $\Lambda$ a subgroup. Let $H\subset\Gamma$ be a w-normal subgroup. Assume that there does not exist a finite index subgroup $H_0$ of $H$ which is contained in a conjugate $g^{-1}\Lambda g$ of $\Lambda,$ for some $g\in \Gamma.$ Assume either that $H$ has the relative property (T) or that $\Lambda$ is amenable and there exists a non-amenable subgroup of $\Gamma$ which commutes with $H$.\\
Let $\sigma_{0}$ be a pmp action of $\Lambda$ on a standard probability space $(X_{0},\mu_0)$ and $\sigma$ the coinduced action of $\Gamma$ on $X:=X_{0}^{\Gamma/\Lambda}$. If $\Gamma\overset{\sigma}{\curvearrowright} X$ is free, then it is OE-superrigid.
\end{corollary}

We need in Corrolary \ref{OEsuper} the freeness assumption of the coinduced action since the proof uses Proposition \ref{superrigid}.                  
See Lemma \ref{freecoinduce} for a large class of coinduced actions that are free. In particular, if $\cap_{g\in\Gamma}g\Lambda g^{-1}=\{e\}$ and $(X_0,\mu_0)$ is non-atomic, then $\Gamma\curvearrowright X$ is free. 

Corrolary\ref{OEsuper} proves for example that any coinduced action of $SL_3(\mathbb Z)$ from a cyclic subgroup is OE-superrigid. We contrast this with Bowen's OE-flexibility results for coinduced actions \cite{B10}. In particular, he proved that any two coinduced actions of $\mathbb F_2=\mathbb Z*\mathbb Z$ from one of the copies of $\mathbb Z$ are OE. Thus, any coinduced action of $\mathbb F_2$ from one of the copies of $\mathbb Z$ is not OE-superrigid.



\subsection{Applications to W$^*$-superrigidity}

For every measure preserving action $\Gamma\curvearrowright X$ of a countable group $\Gamma$ on a standard probability space, we associate {\it the group measure space von Neumann algebra} $L^{\infty}(X)\rtimes\Gamma$ \cite{MvN36}. If the action $\Gamma\curvearrowright X$ is free, ergodic and pmp, then $L^{\infty}(X)\rtimes\Gamma$ is a II$_1$ factor which contains $L^\infty(X)$ as a {\it Cartan subalgebra}, i.e. a maximal abelian von Neumann algebra whose normalizer generates $L^{\infty}(X)\rtimes\Gamma$. 

Two pmp actions $\Gamma\curvearrowright(X,\mu)$ and $\Lambda\curvearrowright (Y,\nu)$ on two standard probability spaces $(X,\mu)$ and $(Y,\nu)$ are said to be W$^*${\it -equivalent} if $L^{\infty}(X)\rtimes\Gamma$ is isomorphic with $L^{\infty}(Y)\rtimes \Lambda.$ It can be seen that orbit equivalence is stronger than W$^*$-equivalence. Moreover, Singer proved in \cite{Si55} that two free ergodic pmp actions $\Gamma\curvearrowright(X,\mu)$ and $\Lambda\curvearrowright (Y,\nu)$ are orbit equivalent if and only if they are W$^*$-equivalent via an isomorphism which identifies the Cartan subalgebras $L^\infty(X)$ and $L^\infty(Y).$
The action $\Gamma\curvearrowright(X,\mu)$ is W$^*$-{\it superrigid} if whenever $\Lambda\curvearrowright(Y,\nu)$ is a free ergodic measure preserving action W$^*$-equivalent with $\Gamma\curvearrowright(X,\mu)$, then the two actions are conjugate.
Therefore, W$^*$-superrigidity for an action $\Gamma\curvearrowright X$ integrates two different rigidity aspects, which are hard to obtain: OE-superrigidity and {\it uniqueness of group measure space Cartan subalgebras}. The latter means that
whenever $M=L^{\infty}(X)\rtimes\Gamma=L^{\infty}(Y)\rtimes \Lambda$, then the Cartan subalgebras $L^\infty(X)$ and $L^\infty(Y)$ are unitarily conjugate in $M.$

A few years ago, the first example of {\it virtually} W$^*$-superrigid actions (i.e. conjugacy is obtained up to finite index subgroups) were found in \cite{Pe09} building on results of \cite{Io08,OP08}. Soon after, Popa and Vaes discovered the first concrete families of W$^*$-superrigid actions \cite{PV09} and Ioana proved that Bernoulli actions of icc property (T) groups are W$^*$-superrigid \cite{Io10}. Subsequently, several other classes of W$^*$-superrigid actions have been found in \cite{FV10,CP10,HPV10,Io10,IPV10,Va10b,CS11,CSU11,PV11,PV12,Bo12,CIK13,CK15}.
By applying Theorems \ref{A} and \ref{B} we will deduce W$^*$-superrigidity for a large class of coinduced actions. To obtain these examples, we will use several results in the literature which prove uniqueness of group measure space Cartan subalgebras for various classes of groups.
 
We denote by $\mathcal {C}$ the class of all countable groups $\Gamma $ which satisfy one of the following conditions:
\begin{enumerate}
\item \cite{CP10} $\Gamma=\Gamma_1\times \Gamma_2$, where $\Gamma_i$ is icc and admits an unbounded cocycle into a mixing representation and a non-amenable icc subgroup with the relative property (T), for $i\in\{1,2\}$;

\item \cite{PV11, PV12} $\Gamma=\Gamma_1\times\Gamma_2\times...\times \Gamma_n$ is a finite product of non-elementary hyperbolic groups with $n\ge 2$;

\item\cite{Io12b}
$\Gamma$ is a finite product of groups of the form $\Gamma_1*_\Sigma\Gamma_2$, each one of them satisfying:
\begin{itemize}
 \item $[\Gamma_1: \Sigma]\ge 2$, $[\Gamma_2: \Sigma]\ge 3$;
 \item there exist $g_1,g_2,...,g_n\in\Gamma$ such that $\cap_{i=1}^{n}g_i\Sigma g_i^{-1}$ is finite. \\
In addition, we assume than one of the factors  $\Gamma^0_1*_{\Sigma^0}\Gamma^0_2$ of $\Gamma$ satisfies the conditions: $\Gamma^0_1$ has property (T) and $\Sigma^0$ is a normal subgroup of $\Gamma^0_2.$
\end{itemize}

\end{enumerate}

If $\Gamma\in \mathcal C$ satisfies condition $(i)$, we say that $\Gamma\in \mathcal{C}_i$, whenever $i\in\{1,2,3\}.$ For $\Gamma\in \mathcal C,$ we fix a subgroup $ \Lambda$ satisfying the following: 
\begin{enumerate}
\item If $\Gamma\in \mathcal C_1$, take $\Lambda$ an amenable subgroup of $\Gamma_1$;
\item If $\Gamma\in \mathcal C_2,$ take $\Lambda$ an amenable subgroup of one of the factors which appears in $\Gamma;$
\item If $\Gamma\in \mathcal C_3,$ take $\Lambda$ such that $\Sigma^0$ does not have a finite index subgroup which is contained in a conjugate of $\Lambda$ (e.g. $\Lambda$ can be taken to be the commutant of $\Sigma^0$ in $\Gamma^0_2$).

\end{enumerate}

Theorems \ref{A} and \ref{B} combined with \cite[Corollary 5.3]{CP10}\cite[Theorem 1.1]{PV12}\cite[Theorem 1.1]{Io12b} give us the following W$^*$-superrigidity result.

\begin{corollary}\label{WW*} Let $\Gamma\in\mathcal C$ a group with no non-trivial finite normal subgroups and $\Lambda$ a subgroup chosen as before. Let $\Lambda\curvearrowright X_0$ be a pmp action on a standard probability space $X_0$ and let $\Gamma\curvearrowright X$ be the coinduced action of $\Lambda\curvearrowright X_0.$ If $\Gamma\curvearrowright X$ is free, then it is W$^*$-superrigid.
\end{corollary}

\begin{example} 
If we take $\Gamma=\Gamma_1*_\Sigma (\Sigma\times\Lambda)\in\mathcal C_3,$ Corollary \ref{WW*} gives another proof of W$^*$-superrigidity for the coinduced action proved in \cite[Example 6.9]{PV09}.
\end{example}


\subsection{Acknowledgements} 

I am very grateful to my advisor Adrian Ioana for suggesting this problem to me and for all the help given through many valuable discussions, important advice and great support. I would also like to thank R\'{e}mi Boutonnet for important remarks and to Daniel Hoff for helpful comments about the paper.

\section {Preliminaries and cocycle rigidity}

At the beginning of this section we review some basic tools of Popa concerning cocycles and continue by introducing the free malleable deformation for Bernoulli actions. The last point will be a cocycle rigidity result of Popa adapted to the context of the free malleable deformations.

\subsection {Notations}

Recall that a {\it tracial von Neumann algebra $(A,\tau_0)$} is a von Neumann algebra $A$ endowed with a normal faithful tracial state $\tau_0.$ We denote by $\|x\|_{2}=\tau_0(x^*x)^{\frac{1}{2}}$ the induced Hilbert norm on $A$, for an element $x\in A$. A von Neumann algebra is finite if and only if is tracial.

\begin{itemize} 
\item We denote by $\oplus_{i\in I}H_i$, the Hilbert direct sum of a family $\{H_i\}_{i\in I}$ of Hilbert spaces and by $H_1\otimes H_2$, the Hilbert tensor product of two Hilbert spaces $H_1$ and $H_2$. 
\item Let $I$ be a nonempty set and $(A,\tau_0)$ be a tracial von Neumann algebra. We denote by $A^{I}$, the tensor product $\bar\otimes_{I}A$, which is again a tracial von Neumann algebra. If we take $a_{i}\in A, i\in I$, such that $\{ i\in I| a_{i}\neq 1 \}$ is finite, we use the notation $\otimes_{i\in I}a_{i}$  for an elementary tensor. We denote by {\it supp}$(\otimes_{i\in I}a_i)$ the set $\{ i\in I| a_{i}\neq 1 \}$, called the support of $\otimes_{i\in I}a_{i}$. $A^I$ has the trace $\tau_0^I$ given by $\tau_0^I(\otimes_{i\in I}a_{i})=\prod_{i\in I}\tau_0(a_i),$ where $\otimes_{i\in I}a_{i}$  is an elementary tensor.


\item For $J\subset I$, we have the canonical embedding $A^{J}\subset A^{I}$, which takes an elementary tensor $\otimes_{j\in J}a_{j}$ to $\otimes_{i\in I}a_{i}$, where $a_i=1$ if $i\notin J$.

\end{itemize}

\subsection{Perturbation of cocycles, property (T) and extensions}

Let $\sigma$ be a trace preserving action of $\Gamma$ on a tracial von Neumann algebra $P$. A map $w:\Gamma\to\mathcal{U}(P)$ is called a $cocycle$ if $w_{gh}=w_g\sigma_g(w_h)$, for all $g,h\in \Gamma.$ Two cocycles $w,w':\Gamma\to\mathcal{U}(P)$ are called $cohomologous$ if there exists a unitary $v\in P$ such that $w_g \sigma_g(v)=vw'_g$, for all $g\in \Gamma.$

\begin{lemma}\label{ext} (\cite[Lemma 2.12]{Po05})
Let $w, w'$ be cocycles for a trace preserving action $\sigma$ of a group $\Gamma$ on a tracial von Neumann algebra $Q$. The following statements are true:
\begin{enumerate}
\item If there exists $\delta>0$ such that $\|w_{g}-w'_{g} \|_{2}\leq \delta$, for all $g\in \Gamma$, then there exists a partial isometry $v\in Q$ such that $\|v-1 \|_{2}\leq 4\delta^{1/2}$ and $w_{g}\sigma_{g}(v)=vw_{g}',$ for all $ g\in\Gamma.$
\item If for any $\epsilon>0$ there exists $u\in \mathcal{U}(Q)$ such that $\|w_{g}\sigma_{g}(u)-uw_{g}'\|_{2}\leq \epsilon,$ for all $ g\in\Gamma,$ then $w$ and $w'$ are cohomologous.
\item If $w$ and $w'$ are cohomologous and $v\in Q$ is a partial isometry satisfying $w_{g}\sigma_{g}(v)=vw_{g}', \forall g\in\Gamma$, then there exists $u\in\mathcal{U}(Q)$ such that $uv^{*}v=v$ and $w_{g}\sigma_{g}(u)=uw_{g}', \forall g\in\Gamma.$
\end{enumerate}

\end{lemma}

Let $\sigma$ be a trace preserving action of a countable group $\Gamma$ on a tracial von Neumann algebra $Q$. Take $w:\Gamma\to\mathcal{U}(Q)$ a cocycle. Let  $\delta$ be a positive real number and a finite subset $F$ of $\Gamma.$ Denote $\Omega_w(\delta,F)=\{w:\Gamma\to\mathcal{U}(Q)| \|w_g-w'_g\|_{2}\leq \delta, \forall g\in F\}$. Assuming this context, we have the following result:
\begin{lemma}\label{neighborhood}(\cite[Lemma 4.2]{Po05})
Let $H\subset\Gamma$ be a subgroup with the relative property (T). Then for every cocycle $w:\Gamma\to\mathcal{U}(Q)$ and $\epsilon>0$, there exist $\delta>0$ and $F$ a finite subset of $\Gamma$ such that $\forall w'\in \Omega_w(\delta,F)$, $\exists v\in Q$ partial isometry satisfying $\|v-1\|_{2}\leq \epsilon$ and $w'_{h}\sigma_{h}(v)=vw_{h}, \forall h\in H.$

\end{lemma}

\begin{definition}Let $\Gamma$ be a countable group and $\sigma$ be a trace preserving action on a tracial von Neumann algebra $(P,\tau).$ The action $\sigma$ is {\it weak mixing} if for every $\epsilon>0$ and finite subset $F$ of $P\ominus\mathbb{C}$, there exists $g\in \Gamma$ such that $|\tau(y^*\sigma_g(x)|\leq\epsilon$, for all $x,y\in F.$
\end{definition}

\begin{proposition}\label{normal}(\cite[Proposition 3.6]{Po05})
Let $\sigma$ and $\sigma '$ be trace preserving actions of a countable group $\Gamma$ on tracial von Neumann algebras $P$ and $N$ and let $w$ be a cocycle for $\sigma\otimes\sigma '$. Let $H\subset \Gamma$ be an infinite normal subgroup and assume that $\sigma$ is weak mixing on $H$. If $w_h\in N,$ for all $h\in H$, then $w_g\in N,$ for all $g\in \Gamma.$

\end{proposition}

\subsection{Coinduced actions for tracial von Neumann algebras and the free product deformation}\label{freeproduct}

The coinduced action for tracial von Neumann algebras is defined as in Section \ref{results}. More precisely, let $\Gamma$ be a countable group and let $\Lambda$ be a subgroup. Let $\phi:\Gamma/\Lambda \to \Gamma$ be a section. Define the cocycle $c:\Gamma\times\Gamma/\Lambda\to \Lambda$ by the formula $$c(g,x)=\phi^{-1}(gx)g\phi(x),$$ for all $g\in\Gamma$ and $x\in \Gamma/\Lambda.$\\
Let $\Lambda\overset{\sigma_{0}}{\curvearrowright} (A,\tau_0)$ be a trace preserving action, where $(A,\tau_0)$ is a tracial von Neumann algebra. We define an action $\Gamma\overset{\sigma}{\curvearrowright} A^{\Gamma/\Lambda}$, called the coinduced action of $\sigma_0$, as follows:
$$\sigma_{g}((a_{h})_{h\in \Gamma/\Lambda})=(a'_{h})_{h\in \Gamma/\Lambda},$$
where $a'_{h}=c(g,h)a_{g^{-1}h}$.\\
Note that $\sigma$ is a trace preserving action of $\Gamma$ on the tracial von Neumann algebra $A^{\Gamma/\Lambda}$.

\begin{remark}\label{abelian}
Let $\Lambda\overset{\sigma_{0}}{\curvearrowright} (X_0,\mu_0)$ be a pmp action, where $(X_{0},\mu_0)$ is a standard probability space. We consider the associated action of $\Lambda$ on $\L^\infty(X_0,\mu_0)$. On one hand, we obtain an coinduced action $\Gamma\overset{\sigma}{\curvearrowright} L^\infty(X_{0},\mu_0)^{\Gamma/\Lambda}$. We also call $\sigma$, the associate action of $\Gamma$ on $X_0^{\Gamma/\Lambda}.$	
Note that $\sigma$ is precisely the usual coinduced action of $\Gamma$ obtained from the action of $\Lambda$ on $X_0.$
\end{remark}

In \cite{Io06a}, Ioana introduced a malleable deformation for general Bernoulli actions, where the base is any tracial von Neumann algebra. This is a variant of the malleable deformation discovered by Popa \cite{Po03} in the case of Bernoulli actions with abelian or hyperfinite base. Here we adapt the deformation of \cite{Io06a} to the context of general coinduced actions.

Let $\Gamma$ be a countable group and $\Lambda$ be a subgroup. Let $A$ be a tracial von Neumann algebra and  $\Lambda\overset{\sigma_{0}}{\curvearrowright} A$ be a trace preserving action. Take $\Gamma\overset{\sigma}{\curvearrowright} A^{\Gamma/\Lambda}$ the corresponding coinduced action. Let $\sigma'$ be a trace preserving action of $\Gamma$ on another tracial von Neumann algebra $(N,\tau')$. \\
Denote by $\tilde{A}$ the tracial von Neumann algebra $A*L(\mathbb{Z})$, which is the free product of $A$ and $L(\mathbb{Z})$. Take $u\in L(\mathbb{Z})$ the canonical generating Haar unitary. Let $h=h^{*}\in L(\mathbb{Z})$ be such that $u=$ exp$(ih)$ and set $u_{t}=$ exp$(ith)$ for all $t\in \mathbb{R}.$ Denote by $P=A^{\Gamma/\Lambda}$ and $\tilde{P}=\tilde{A}^{\Gamma/\Lambda}$ the tensor product von Neumann algebras and define $\theta : \mathbb{R}\to$ Aut$(\tilde{P})$ by
\begin{center}
$\theta_{t}(\otimes_{h\in \Gamma/\Lambda}a_{h})=\otimes_{h\in \Gamma/\Lambda}$ Ad$(u_{t})(a_{h}),$\\
\end{center}
where $\otimes_{h\in \Gamma/\Lambda}a_{h}\in \tilde{P}$ is an elementary tensor. 

We observe that $\theta_t$ extends naturally as an automorphism of $\tilde P \bar\otimes N$. Define also $\beta\in$ Aut$ (\tilde P\bar\otimes N)$ by $\beta_{|P\bar\otimes N}=id_{P\bar\otimes N}$ and $\beta(\otimes_{h\in F}u)=\otimes_{h\in F}u^*,$ for all finite subsets $F$ of $\Gamma/\Lambda.$ \\
Notice that the action $\sigma$ extends naturally to an action $\tilde\sigma$ on $\tilde{P}$ by letting $\tilde\sigma_{g}(\otimes_{h\in F}u)=\otimes_{h\in F}u$, for all finite subsets $F$ of $\Gamma/\Lambda$. We denote by $\rho$ the tensor product action $\sigma\otimes\sigma'$ of $\Gamma$ on $P\bar\otimes N$ and by $\tilde\rho$ the tensor product action $\tilde\sigma\otimes\sigma'$ of $\Gamma$ on $\tilde P\bar\otimes N$.  

\begin{remark}
Notice that $\tilde\rho$ commutes with the automorphims $\beta$ and $\theta_{t}$ for all $t$. Thus, we can consider $\beta$ and $\theta_{t}$ as automorphisms of $(P\bar\otimes N)\rtimes \Gamma$ and $(\tilde P\bar\otimes N)\rtimes\Gamma $, by extending them in a natural way. Also note that $\beta\theta_t=\theta_{-t}\beta$ and $\beta^{2}=id.$
\end{remark}

\subsection{Finite union of translates of a subgroup and a fixed point lemma}

\begin{lemma}\label{finiteindex}
Let $H$ be a group and $H_i$ subgroups, for $1\leq i\leq n$. Suppose that there exist finite subsets $F_i$ of $ H$ such that
$$H=\cup_{i=1}^{n}F_iH_i.$$ 
Then there exists $i\in\{1,2,...,n\}$ such that $H_i$ is a subgroup of finite index in $H$.
\end{lemma}

\begin{proof}
We will proceed by induction over $n$. For $n=1$ it is clear. Let us suppose the statement is true for $n-1$ and prove it for $n.$
We consider the case where $H_n$ is a subgroup of infinite index in $H$, otherwise we are done. 

Let us write a partition of $H$ via the infinite index subgroup $H_n$: 
$$H=F_n H_n \cup(\cup_{k=1}^{\infty}h_k H_{n}),$$
where $h_{j}^{-1}h_i\notin H_{n},$ for all $i\neq j$ and $h_{k}^{-1}h_{0}\notin H_{n}$ for all $k\ge 1$ and $h_{0}\in F.$

Then, $\cup_{i=1}^{\infty}h_i H_{n}\subset\cup_{i=1}^{n-1}F_iH_i.$ Since $H$ can be written as finite union of translates of $\cup_{i=1}^{\infty}h_i H_{n}$, we obtain that $H$ can be also written as finite union of translates of $\cup_{i=1}^{n-1}F_iH_i.$ Thus,
$$H=\cup_{i=1}^{n-1}F'_iH_i,$$
with $F_i'$ some finite subsets of $H.$
Now we can apply the induction hypothesis and conclude that at least one of the $H_i$'s is a subgroup of finite index in $H$ for an $i\in \{1,2,...,n\}.$
\end{proof}

\begin{remark}
The following proposition is a consequence of \cite[Lemma 2.4]{PV06}, but we include a proof for the reader's convenience. 
\end{remark}

\begin{proposition}\label{finiteunion}
Let $\Gamma$ be a countable group and $\Lambda$ a subgroup. Let $H$ be another subgroup of $\Gamma.$ Then there exists a finite set $F\subset \Gamma/\Lambda$ such that $gF\cap F\neq\emptyset$ for all $g\in H$ if and only if there exists a subgroup $H_0$  of finite index of $H$ such that $H_0$ is contained in a conjugate $g^{-1}\Lambda g$ of $\Lambda.$
\end{proposition}
\begin{proof}
Let us suppose that there exists a finite set $F\subset \Gamma/\Lambda$ such that $gF\cap F\neq\emptyset$ for all $g\in H$. Let $F=\{f_1, f_2,\dots f_n\}$. Then for all $h\in H$, there exist $i,j\in\{1,2,...,n\}$ such that $hf_{j}\Lambda=f_{i}\Lambda.$
We obtain that $H\subset\cup_{i,j=1}^n f_i \Lambda f_j^{-1}.$

Let $H_{ij}:=\{h\in H| hf_{j}\Lambda=f_{i}\Lambda\}$ and notice that $H=\cup_{i,j=1}^{n}H_{ij}.$ For $i\neq j$, if $H_{ij}\neq\emptyset, $ take $g_{ij}\in H_{ij}$ an arbitrary element. Observe that $H_{ij}=g_{ij}H_{jj}.$  For $i\neq j$, if $H_{ij}=\emptyset, $ choose $g_{ij}$ to be the neutral element. This allows us to write $H$ in the form $H=\cup_{i,j=1}^{n}g_{ij}H_{jj}$ which is sufficient for applying Lemma \ref{finiteindex}, where $g_{ii}$ is the neutral element for all $i\in \{1,2,...,n\}.$ 

Notice that $H_{ii}=H\cap f_{i}\Lambda f_{i}^{-1}$ and at least one of these subgroups is of finite index in $H$ because of Lemma \ref{finiteindex}. 

The converse is easy. This finishes the proof.
\end{proof}

For the following lemma we use the notations from Section \ref{freeproduct}.
\begin{lemma}\label{fix}
Let $H$ be a subgroup of $\Gamma$. Assume that there does not exist a subgroup $H_0$ of finite index in $H$ such that $H_0$ is contained in a conjugate $g^{-1}\Lambda g$ of $\Lambda.$ Let  $w_{h}$ and $w'_{h}$ be arbitrary elements in $P\bar\otimes N$, for all $ h\in H$, and define the map $\alpha:H \to\mathbb{B}( L^{2}(\tilde{P}\bar\otimes N))$  by $\alpha_{h}(x)=\gamma(w'_{h})\tilde\rho_{h}(x)w_{h}$, where $\gamma\in \{ id, \theta_{1}\}$. Let $S$ be the $\|\cdot\|_{2}$-closed linear subspace of $\tilde{P}$ generated by $\gamma(P)P$. 
Then
$$\{\xi\in \tilde{P}\bar\otimes N)|\alpha_h(\xi)=\xi, \forall h\in H\}\subset S\otimes L^2(N).$$
\end{lemma}

\begin{proof} We begin the proof with a claim which will prove the lemma.

{\bf Claim.} For any $\epsilon>0$ and $\xi, \eta \in \tilde{P}\bar\otimes N$ with $\xi, \eta \perp S\bar\otimes N$, there exists $h\in H$ such that  $$|\langle\xi,\alpha_{h}(\eta)\rangle|\leq \epsilon \|\xi\|_{2}\|\eta\|_{2}.$$
To prove the claim, we can assume $\|\xi\|_{2}=\|\eta\|_{2}=1$. Let us take $\xi_{0}, \eta_{0} \in \tilde{P}\bar\otimes N$ with $\|\cdot\|_2$ norm smaller than 1 and $F$ a finite subset of $\Gamma/\Lambda$ such that
$$\|\xi-\xi_{0}\|_{2}\leq \epsilon/2,\quad \xi_{0}=\sum_{i=1}^{n}p_{i}\otimes n_{i},\quad p_i\in \tilde A^F\subset\tilde P,\quad n_i\in N, \quad p_{i}\perp S, \forall i\in \{1,2,...,n\}.$$
and 
$$\|\eta-\eta_{0}\|_{2}\leq \epsilon/2,\quad \eta_{0}=\sum_{i=1}^{n}q_{i}\otimes m_{i}, \quad q_i\in \tilde A^F\subset\tilde P, \quad m_i\in N, \quad q_{i}\perp S, \forall i\in \{1,2,...,n\}.$$
Proposition \ref{finiteunion} allows us to take $h\in H,$ such that $hF\cap F=\emptyset$. By the triangle inequality we have
$$
\begin{array}{rcl} |\langle\xi,\alpha_{h}(\eta)\rangle|&\leq& |\langle\xi-\xi_{0},\alpha_{h}(\eta)\rangle|+|\langle\xi_0,\alpha_{h}(\eta-\eta_{0})\rangle|+|\langle\xi_{0},\alpha_{h}(\eta_{0})\rangle|\\
&\leq& \epsilon/2+\epsilon/2+|\langle\xi_{0},\alpha_{h}(\eta_{0})\rangle|.
\end{array}
$$
We will prove the claim if we show that $$\langle\xi_{0},\alpha_{h}(\eta_{0})\rangle=\langle\xi_{0},\gamma(w'_h)\tilde\rho_{h}(\eta_{0})w_h\rangle=0.$$ By linearity and continuity (weak operator topology) we may suppose that $w_h=\otimes_{F'}a_{j}\otimes n, w'_h=\otimes_{F'}a'_{j}\otimes n'\in P\bar\otimes N$ are elementary tensors with $F'\subset \Gamma/\Lambda$ a finite subset and $\otimes_{F'}a_j, \otimes_{F'} a'_j\in A^{\Gamma/\Lambda}=P,  n,n'\in N$. By the above we may assume that $\xi_{0}=p_{0}\otimes n_{0}, \eta_{0}=q_{0}\otimes m_{0}\in \tilde P\bar\otimes N$, $p_0$ and $q_0$ orthogonal to $S$ and $n_0, m_0\in N$. Moreover, $p_0$ and $q_0$ can be considered to have support contained in $F$.

This scalar product will be proven to be $0$ by computing it more explicitly. First, notice that the support of the elements from $\tilde P$ which appear in the scalar product is contained in $F\cup hF\cup F'$. Denote by $\tilde\tau$  the trace on $\tilde P$. Then, since $F\cap hF=\emptyset$,
we have the decomposition
$$\langle\xi_{0}, \gamma(w')\rho_{h}(\eta_{0})w\rangle =\tilde\tau(b_1)\tilde\tau(b_2),$$
where $b_1=\otimes_{F\cap F'}a_{j}^{*}\gamma({a'}_{j}^{*})  p_{0}\in \tilde A^{F}$ and $b_2\in \tilde A^{(hF\cup F')\setminus F}\bar\otimes N.$

The first factor is $0$ because $p_{0}$ is orthogonal to $S.$ This proves the claim. \hfill$\square$

Now, we can finish the proof of the lemma. Take $v\in \tilde{P}\bar\otimes N$ such that $\alpha_h(v)=v$, for all $h\in H$. Write $v=v_{0}+v_{\perp}$ with $v_{0}\in S\otimes L^2(N)$ and $v_{\perp} \perp S\otimes L^2(N)$. Since $S\otimes L^2(N)$ is $\alpha$-invariant, we get that $v_{0}$ and $v_{\perp}$ are $\alpha$-invariant. The claim gives us that $v_\perp=0,$ which implies that $v\in S\bar\otimes N$. This ends the lemma.

\end{proof}

\subsection{Cocycle rigidity}

The following proposition is the first part of \cite[Proposition 3.2]{Po05}. Before writing the result, let us introduce some terminology.

Let $\Gamma$ be a countable group and $\sigma$ be a trace preserving action of $\Gamma$ on a tracial von Neumann algebra $Q$. We recall that a $local\; cocycle$ for the action $\sigma$ is a map $w$ on $\Gamma$ with values in the set of partial isometries of $Q$ which satisfies $w_g\sigma_g(w_h)=w_{gh}$, for all $g,h\in \Gamma.$

Let $\sigma'$ be a trace preserving action of $\Gamma$ on another tracial von Neumann algebra $N$ and denote by $\rho$ the tensor product action $\sigma\otimes\sigma'$. For a cocycle $w:\Gamma\to \mathcal{U}(Q\bar\otimes N)$ we denote by $w^l:\Gamma\to\mathcal{U}(Q\bar\otimes Q\bar\otimes  N)$ the image of $w$ via the canonical isomorphism and inclusion $Q\bar\otimes N\simeq Q\bar\otimes 1\bar\otimes N\subset Q\bar\otimes Q\bar\otimes N.$ Similarly, we denote by $w^r$ the image of $w$ via the canonical  isomorphism and inclusion $Q\bar\otimes N\simeq 1\bar\otimes Q\bar\otimes N\subset Q\bar\otimes Q\bar\otimes N.$


\begin{proposition}\cite[Proposition 3.2]{Po05}\label{popa}
Let $\sigma$ be a weak mixing trace preserving action of $\Gamma$ on a tracial von Neumann algebra $Q$ and $\sigma'$ a trace preserving action of $\Gamma$ on another tracial von Neumann algebra $N$.  Let $w:\Gamma\to \mathcal{U}(Q\bar\otimes N)$ be a cocycle for the action $\rho $. Let $b\in L^{2}(Q\bar\otimes Q\bar\otimes N)$ be a non-zero element and $p\in \mathcal{P}(Q\bar\otimes 1\bar\otimes N)$ be a non-zero projection such that $pb=b$ and $w_{g}^{l}\bar\sigma_{g}(b){w_{g}^{r}}^*=b$, for all $g\in\Gamma,$ where $\bar\sigma:=\sigma\otimes \sigma \otimes \sigma'$. Then, there exist a partial isometry $v\in Q\bar\otimes N$  and a local cocycle $w'_{g}\in\mathcal {U}(v^*v N \sigma_{g}'(v^*v))$  such that $vv^*\leq p, v^*v\in N$ and  $w_{g}(\sigma\otimes\sigma'_{g})(v)=vw'_{g}, $ for all $g\in \Gamma.$
\end{proposition}

\begin{remark}
Let us explain why the first part of \cite[Proposition 3.2]{Po05} can be written as above. 
\begin{itemize}
\item In \cite{Po05} the tracial von Neumann algebra $(Q,\tau)$ is extended to a larger tracial von Neumann algebra $(\tilde Q,\tilde \tau)$ satisfying the following properties: it exists a trace preseving action $\tilde\sigma$ of $\Gamma$ on $\tilde Q$ which extends $\sigma$ and an automorphism $\alpha_1$ of $\tilde Q$ which satisfies $\overline{sp}^w\; Q\alpha_1(Q)=\tilde Q$ and $\tilde\tau(x\alpha_1(y))=\tau(x)\tau(y)$, for all $x,y\in Q.$\\
In particular, it can be chosen $\tilde Q=Q\bar\otimes Q.$
\item Notice that $b$ can be chosen in $L^{2}(\tilde Q\bar\otimes N)$ in \cite[Proposition 3.2]{Po05}, not necessary in $\tilde Q\bar\otimes N$, since the proof uses only this information.

\end{itemize}

\end{remark}

From now on until the end of the section, we assume the following context. Let $\Lambda$ be a subgroup of a countable  group $\Gamma.$ Let $\sigma_{0}$ be a trace preserving action of $\Lambda$ on a tracial von Neumann algebra $A$ and $\sigma$ the coinduced action of $\Gamma$ on $P:=A^{\Gamma/\Lambda}$. Let us consider a trace preserving action $\sigma'$ of $\Gamma$ on another tracial von Neumann algebra $N$. 

Denote by $\rho$ the tensor product action $\sigma\bar{\otimes}\sigma '$ of $\Gamma$ on $P\bar{\otimes}N$, by $\tilde\rho$ the tensor product action $\tilde\sigma\otimes \sigma'$ of  $\Gamma$ on $\tilde P\bar\otimes N$ and by $\bar\sigma$ the tensor product action $\sigma\otimes\sigma\otimes\sigma'$ of $\Gamma$ on $P\bar\otimes P\bar\otimes N.$

Let $w:\Gamma\to\mathcal{U}(P\bar\otimes N)$ be a cocycle for $\rho.$
Define the representations $\pi:\Gamma\to \mathcal{U}(L^{2}(P\bar\otimes P \bar\otimes N))$ and $\gamma:\Gamma\to \mathcal{U}(\overline{sp}\;P\theta_{1}(P)\otimes L^{2}(N))$, by $\pi_{g}(b)=w_{g}^l\bar\sigma_{g}(b){w_{g}^{r}}^*$ and $\gamma_{g}(c)=w_{g}\tilde\rho_{g}(c)\theta_{1}(w_{g})^*.$ Here we have denoted by $\overline{sp}\;P\theta_{1}(P)$ the $\|\cdot\|_{2}$-closed linear subspace generated by $\{x\theta_1(y)|x,y\in P\}$.

Notice that $L^{2}(P\bar\otimes P \bar\otimes N)$ and $\overline{sp}\;P\theta_{1}(P)\otimes L^{2}(N)$ may be viewed as left $P\bar\otimes N$ Hilbert modules with the actions $(p\otimes n)\cdot (x\otimes y\otimes n'):=px\otimes y\otimes nn'$ and, respectively, $(p\otimes n)\cdot x\theta_1(y)\otimes n':=px\theta_1(y)\otimes nn',$ for all $p,x,y\in P$ and $n,n'\in N.$
The following lemma makes Proposition \ref{popa} useful in our context in which we work with the free product deformation. The proof is a straightforward verification.

\begin{lemma}\label{isom}
The map $U:L^{2}(P\bar\otimes P \bar\otimes N)\to \overline{sp}\;P\theta_{1}(P)\otimes L^{2}(N)$ defined by $U(p_{1}\otimes p_{2}\otimes n)=p_{1}\theta_{1}(p_{2})\otimes n$, with $p_{1}, p_{2}\in P, n\in N,$ is an isomorphism of  Hilbert spaces which intertwines the representations $\pi$ and $\gamma.$ Moreover, $U$ intertwines the left $P\bar\otimes N$ - module structures of these Hilbert spaces.
\end{lemma}

In order to apply Proposition \ref{popa}, we need the weak mixing property for the coinduced action. 	

\begin{lemma}\label{weakmixing} Let $H$ be a subgroup of $\Gamma$ with the property that there is no finite index subgroup $H_0$ of $H$ which is contained in a conjugate $g\Lambda g^{-1}$ of $\Lambda$. Then the coinduced action $\sigma$ is weak mixing on $H$.
\end{lemma}

Ioana proved this result for coinduced actions on standard probability spaces in \cite[Lemma 2.2]{Io06b}, but the proof also works for tracial von Neumann algebras.

Using the same arguments as  in the second part of the proof of  \cite[Proposition 3.2]{Po05}, we obtain the following result: 

\begin{theorem}\label{final}
Let $\Gamma$ be a countable group and $\Lambda$ be a subgroup.  Let $H$ be a subgroup of $\Gamma$ with the property that there is no finite index subgroup $H_0$ of $H$ such that $H_0$ is contained in a conjugate $g\Lambda g^{-1}$ of $\Lambda$. Let $w:\Gamma\to\mathcal{U}( P\bar\otimes N)$ be a cocycle for the action $\rho$. If $w_{|H}$ and $\theta_{1}(w)_{|H}$ are cohomologous, then $w_{|H}$ is cohomologous with a cocycle with values in $N$.
\end{theorem}

\begin{proof}
We will use Proposition \ref{popa} and a maximality argument.

Denote by $\mathcal {W}$ the set of pairs $(v,w')$ with $v\in P\bar\otimes N $ partial isometry satisfying $v^{*}v\in N$ and $w':\Gamma\to \mathcal{U}(v^*v N \sigma'(v^*v))$ local cocycle for $\rho$ such that $vw'_{g}=w_{g}\rho_{g} (v)$, for all $g\in \Gamma.$

We endow $\mathcal{W}$ with the order: $(v_{0},w'_{0})\leq (v_{1},w'_{1})$ iff $v_{0}=v_{1}v^{*}_{0}v_{0}, v^{*}_{0}v_{0}w'_{1}(g)=w'_{0}(g),$ for all $g\in \Gamma$. $\mathcal{W}$ is an inductive set and let $(v_{0}, w'_{0})\in \mathcal{W}$ be a maximal element.

 {\bf Claim.} $v_{0}$ is a unitary. \\
Note that the claim finishes the proof. Let us prove the claim by contradiction. Suppose $v_{0}$ is not a unitary. Denote by $v=v_{0}\theta_{1}(v_{0}^*)$. Then $vv^*=v_{0}v_{0}^*$ and a direct computation gives us that $w_{g}\tilde\rho_{g}(v)=v\theta_{1}(w_{g}).$ Indeed, since $\rho_{g}(w_{g^{-1}}^*)=w_{g}$ and $\rho_{g}(w'_0(g^{-1})^*)=w'_0(g)$, we have
$$ 
\begin{array}{rcl}
w_{g}\tilde\rho_{g}(v) &=& w_{g}\tilde\rho_{g}(v_{0})\tilde\rho_{g}(\theta_{1}(v_{0}^*))=v_{0}w'_{0}(g)\tilde\rho_{g}(\theta_{1}(v_{0}^*))\\ 
&=& v_{0}\theta_{1}(\tilde\rho_{g}(v_{0}w'_{0}(g^{-1}))^*)=v_{0}\theta_{1}(\tilde\rho_{g}(w_{g^{-1}}\rho_{g^{-1}}(v_0{}))^*)\\
&=& v_{0}\theta_{1}(v_{0}^*\rho_{g}(w_{g^{-1}}^*))=v_{0}\theta_{1}(v_{0}^*w_{g})\\
&=& v\theta_{1}(w_{g}).
\end{array}
$$

Since $w$ and $\theta_{1}(w)$ are cohomologous, by Lemma \ref{ext}, we obtain the existence of a partial isometry $v'\in \tilde P\bar\otimes N$ such that $w_{g}\tilde\rho_{g}(v')=v'\theta_{1}(w_{g})$ and $v'v'^{*}=1-vv^*, v'^*v'=1-v^*v.$ 

Next, Lemma \ref{fix} implies that $v'\in \overline{sp}\;P\theta_{1}(P)\bar\otimes N$, which allows us to use Lemma \ref{isom}. Since $v'$ is a fixed point for $\gamma$, $U^{-1}(v')$ is a fixed point for $\pi$. Now we can apply Proposition \ref{popa} to obtain the existence of a partial isometry $v_{1}\in P\bar\otimes N$ with the left support majorized by $l(U^{-1}(v'))$ and right support in $N$ which satisfies $v_{1}w'_{1}(g)=w_{g}\tilde\rho_g(v_{1})$ for some local cocycle $w_{1}':\Gamma\to \mathcal U(v_{1}v_{1}^*N\sigma'(v_{1}v_{1}^*))$. Here we denote by $l(U^{-1}(v'))$ the left support of $U^{-1}(v').$ 

Notice that $l(U^{-1}(v'))$ is majorized by $v'v'^*=1-v_{0}v_{0}^{*}$. Indeed, by Lemma \ref{isom}, $U$ intertwines the $P\bar\otimes N$ left module structure. Now, since $v'v'^*=1-vv^*=1-v_0v^{*}_0\in P\bar\otimes N$, we have $U^{-1}(v')=U^{-1}(v'v'^*v')=v'v'^*U^{-1}(v')$, which proves the claim.

Thus, in the finite von Neumann algebra $\tilde P\bar\otimes N$ we have $v_{1}^*v_{1}\sim v_{1}v_{1}^*\leq 1-v_{0}v_{0}^* \sim1-v_{0}^{*}v_{0}$. Since the first and the last projection lies in $N$, we obtain that $v_{1}^*v_{1}\preceq 1-v_{0}^*v_{0}$ in $N$ (by working with the central trace).

Now, we conclude as in the proof of \cite[Proposition 3.2]{Po05}. By multiplying $v_{1}$ to the right with a partial isometry in $N$ and conjugate $w_{1}'$ appropriately, we may assume $v_{1}^*v_{1}\leq 1-v_{0}^*v_{0}.$ But then, $(v_{0}+v_{1},w'_{0}+w'_{1})\in \mathcal{W}$ and strictly majorizes $(v_{0},w'_{0})$, which contradicts the maximality assumption.
\end{proof}

\section{Proof of Theorem \ref{A}}

We will prove the following theorem, which is the general version of Theorem \ref{A} dealing with coinduced actions of $\Gamma$ on $A^{\Gamma/\Lambda}$ that arise from actions of $\Lambda$ on arbritrary tracial von Neumann algebras A.

\begin{theorem}[Groups with relative property (T)]\label{thmA} 

Let $\Gamma$ be a countable group and $\Lambda$ be a subgroup. Let $H\subset\Gamma$ be a subgroup with relative property (T). Assume that there does not exist a subgroup $H_0$ of finite index in $H$ such that $H_0$ is contained in a conjugate $g^{-1}\Lambda g$ of $\Lambda.$ 

Let $\sigma_{0}$ be a trace preserving action of $\Lambda$ on a tracial von Neumann algebra $A$ and $\sigma$ the coinduced action on $P:=A^{\Gamma/\Lambda}$. Let us consider another action $\sigma '$ on a tracial von Neumann algebra $N$. Denote by $\rho$ the tensor product action $\sigma\bar{\otimes}\sigma '$ of $\Gamma$ on $P\bar{\otimes}N.$

Then, any cocycle $w:\Gamma\to \mathcal{U}(P\bar{\otimes}N)$ for the restriction of $\rho$ to $H$ is cohomologous with a cocycle of the form $w':H\to \mathcal{U}(N)$.\\
Moreover, if $H$ is w-normal in $\Gamma$, then $w$ is cohomologous with a cocycle of the form $w':\Gamma\to \mathcal{U}(N)$.
\end{theorem}

From now on, in this section we use the same notations as in Section \ref{freeproduct}. The first step of the proof of Theorem \ref{thmA} is to prove that $w_{|H}$ and $\theta_1(w)_{|H}$ are cohomologous. This is obtained by the following result which
is \cite[Lemma 4.6]{Po05} adapted to the free product deformation.

\begin{proposition}\cite[Lemma 4.6]{Po05}\label{theta}
Let $\Lambda$ be a subgroup of $\Gamma$. Let $H\subset\Gamma$ be a subgroup with relative property (T) such that there does not exist a subgroup $H_0$ of finite index in $H$ which is contained in a conjugate $g^{-1}\Lambda g$ of $\Lambda.$ 

Let $\sigma_{0}$ be a trace preserving action of $\Lambda$ on a tracial von Neumann algebra $A$ and $\sigma$ the coinduced action on $P=A^{\Gamma/\Lambda}$. Consider a trace preserving action $\sigma '$ on a tracial von Neumann algebra $N$. 

Let $w:\Gamma\to P\bar{\otimes}N$ be a cocycle for the action $\rho$ on $P\bar\otimes N$. Then $w_{|H}$ and $\theta_{1}(w)_{|H}$ are cohomologous as cocycles for the action $\tilde\rho_{|H}$ on $\tilde{P}\bar\otimes N$.

\end{proposition}

The proof of Proposition \ref{theta} is almost identical to that of \cite[Lemma 4.6]{Po05}, but we include it for completeness. At the end of the proof of \cite[Lemma 4.6]{Po05} it is used the weak mixing property and therefore is obtained that a certain element is in a smaller algebra. The difference is that in the proof of Proposition \ref{theta} is used Lemma \ref{fix} to obtain the same result. 

\begin{proof}[Proof of Proposition \ref{theta}]
It is enough to prove that $\forall \epsilon>0, \exists v\in \tilde{P}\bar\otimes N$ partial isometry such that $\|v^{*}v-1\|_{2}\leq \epsilon$ and
$$w_{h}\tilde\rho_{h}(v)=v\theta_{1}(w_{h}), \forall h\in H.$$

Indeed, if this holds, take a unitary $u\in \mathcal{U}(\tilde{P}\bar\otimes N)$ satisfying $uv^{*}v=v$. By triangle inequality, we get that
$$\|w_{h}\tilde\rho_{h}(u)-u\theta_{1}(w_{h})\|_{2}\leq 2\|u-v\|_{2}=2\|1-v^{*}v\|_{2}\leq 2\epsilon, \forall h\in H.$$
Using now Lemma \ref{ext}, we get that $w_{|H}$ and $\theta_{1}(w)_{|H}$ are cohomologous.

We now prove the first statement of this proof in two steps.

{\bf Step 1.} For all $\epsilon>0$, there exist $v_{0}\in \tilde{P}\bar\otimes N$ and $n\in \mathbb{N}$ such that $\|v_{0}^{*}v_{0}-1\|_{2}\leq \epsilon$ and
\begin{equation}\label{formula}
w_{h}\tilde\rho_{h}(v_{0})=v_{0}\theta_{1/2^{n}}(w_{h}), \forall h\in H.
\end{equation}
This is just an application of Lemma \ref{neighborhood}. Indeed, the lemma gives us the existence of a partial isometry $v_{0}\in \tilde{P}\bar\otimes N$ and $n\in \mathbb{N}$, satisfying $\|v_{0}-1\|_{2}\leq \epsilon/2$ such that formula \ref{formula} holds. Using the triangle inequality, we get that $\|v_{0}^{*}v_{0}-1\|_{2}\leq \epsilon$.

{\bf Step 2.} Assume that there exists a partial isometry $v\in \tilde{P}\bar\otimes N$ and $t\in (0,1)$ satisfying
\begin{equation}\label{formula2}
w_{h}\tilde\rho_{h}(v)=v\theta_{t}(w_{h}), \forall h\in H.
\end{equation}
Then there exists a partial isometry $v'\in \tilde{P}\bar\otimes N$ satisfying $\|v\|_{2}=\|v'\|_{2}$ and
$$w_{h}\tilde\rho_{h}(v')=v'\theta_{2t}(w_{h}), \forall h\in H.$$

For proving Step 2, we will use the properties of the automorphism $\beta$. Since $\beta\theta_{t}=\theta_{-t}\beta$ and $\beta_{| P\bar\otimes N}=id_{P\bar\otimes N}$ we get that
$$w_{h}\tilde\rho_{h}(\beta(v))=\beta(v)\theta_{-t}(w_{h}), \forall h\in H.$$
By taking the adjoint in \ref{formula2}, we obtain 
$$v^{*}w_{h}=\theta_{t}(w_{h})\tilde\rho_{h}(v^{*}), \forall h\in H.$$
Define now $v'=\theta_{t}(\beta(v)^{*}v)$. We get
$$
\begin{array}{rcl} v'^{*}w_{h} &=& \theta_{t}(v^{*}\beta(v)\theta_{-t}(w_{h}))\\
&=& \theta_{t}(v^{*}w_{h}\tilde\rho_{h}(\beta(v))) \\
&=& \theta_{t}(\theta_{t}(w_{h})\tilde\rho_{h}(v^{*}\beta(v)))\\
&=& \theta_{2t}(w_{h})\tilde\rho_{h}(v'^{*}),

\end{array}
$$
which implies that
$$w_{h}\tilde\rho_{h}(v')=v'\theta_{2t}(w_{h}), \forall h\in H.$$

Let us prove now that $\|v\|_{2}=\|v'\|_{2}.$ Since $\|v'\|_{2}=\|\beta(v)^{*}v\|_{2}$, it's enough to prove that $\beta(vv^{*})=vv^{*}$. Using the equation \ref{formula2} and applying the adjoint to it, we get that
$$w_{h}\tilde\rho_{h}(vv^{*})w_{h}^{*}=vv^{*}, \forall h\in H.$$
By Lemma \ref{fix}, we obtain that $vv^{*}\in P\bar\otimes N$, so $\beta(vv^{*})=vv^{*}.$ This ends the proof.
\end{proof}

The proof of Theorem \ref{thmA} is now an easy consequence of Proposition \ref{theta} and Theorem \ref{final}.

{\bf Proof or Theorem \ref{thmA}}

By Proposition \ref{theta}, there exists a unitary $v\in \tilde{P}\bar\otimes N$ such that
$$w_{h}\rho_{h}(v)=v\theta_{1}(w_{h}), \forall h\in H.$$
Theorem \ref{final} gives us the existence of a cocycle $w':H \to \mathcal{U}(N)$ cohomogous with $w$. More precisely, we have
$$w_{h}=uw'_{h}\rho_{h}(u^{*}), \quad \forall h\in H,$$
for a unitary $u\in\mathcal{U}(P\bar\otimes N).$

For the moreover part, notice that Lemma \ref{weakmixing} implies that the coinduced action is weak mixing on $H$. Thus, we can apply Proposition \ref{normal} and obtain that $u^{*}w_{g}\rho_{g}(u)\in N$, for all $g\in\Gamma.$  This allows us to define $w'$ on $\Gamma$ and obtain that $w$ is cohomologous with a cocycle with values in $N$ on $\Gamma.$\hfill$\square$

\begin{remark} Theorem \ref{thmA} implies Theorem \ref{A}. Indeed, this is true by Remark \ref{abelian} and
\cite[Proposition 3.5]{Po05}, which allows us to untwist a cocycle into a $\mathcal U_{fin}$ group once is unwisted into $\mathcal U(N).$ 
\end{remark}

\section{Proof of Theorem \ref{B}}

In this section we prove Theorem \ref{thmB}, which is a more general version of Theorem \ref{B} dealing with coinduced actions of $\Gamma$ on $A^{\Gamma/\Lambda}$ that arise from actions of $\Lambda$ on arbitrary tracial von Neumann algebras $A$.

\begin{theorem}[Product groups]\label{thmB}
Let $\Gamma$ be a countable group and $\Lambda$ be an amenable subgroup. Let $H$ and $H'$ be infinite commuting subgroups of $\Gamma$ such that $H'$ is non-amenable. Assume that $H$ does not have a subgroup $H_0$ of finite index in $H$ such that $H_0$ is contained in a conjugate $g^{-1}\Lambda g$ of $\Lambda.$

Let $\sigma_{0}$ be a trace preserving action of $\Lambda$ on a tracial von Neumann algebra $A$ and $\sigma$ the coinduced action on $P:=A^{\Gamma/\Lambda}$. Let us consider another action $\sigma '$ on a tracial von Neumann algebra $N$. Denote by $\rho$ the tensor product action $\sigma\bar{\otimes}\sigma '$ of $\Gamma$ on $P\bar{\otimes}N.$  

Then, any cocycle $w:\Gamma\to \mathcal{U}(P\bar{\otimes}N)$ for the restriction of $\rho$ to $HH'$ is cohomologous with a cocycle of the form $w':HH'\to \mathcal{U}(N)$.\\
Moreover, if $H$ is w-normal in $\Gamma$, then $w$ is cohomologous with a cocycle of the form $w':\Gamma\to \mathcal{U}(N)$.
\end{theorem}

We use the same notations as in section \ref{freeproduct}. We still consider $\sigma$ the coinduced action on $P$, $\sigma '$ a trace preserving action on a tracial von Neumann algebra $N$ and the free product deformation $\theta_{t}$.
 
The following result is known as Popa's $transvesality$ lemma. 
\begin{lemma}\label{transversality}(\cite[Lemma 2.1]{Po06})
For every $s\in (0,1/2)$ and $x\in P\bar\otimes N$, we have
$$\|\theta_{2s}(x)-x\|_{2}\leq 2\|\theta_{s}(x)-E_{P\bar\otimes N}(\theta_{s}(x))\|_{2}.$$
\end{lemma}

\begin{lemma}\label{amenable}
Let $\Gamma$ be a countable group and $\Lambda$ an amenable subgroup. Let $F$ be a finite subset of $\Gamma/\Lambda.$ Denote $N_F=\{g\in \Gamma| gF=F\}$, where $\Gamma$ acts on $\Gamma/\Lambda$ by left multiplication. Then $N_F$ is amenable.
\end{lemma}

\begin{proof}
The action of $N_F$ on $F$, by left multiplication, is well defined. Denote by $S_F$ the group of bijections on the finite set $F.$ We obtain a homomorphism $\phi: N_F\to S_F$, defined by $\phi(g)f=gf,$ for all $g\in N_F$ and $f\in F$.

 Notice that $ker\phi$, the kernel of $\phi$, is amenable. Indeed, if $f\Lambda\in F$, $ker\phi\subset f\Lambda f^{-1}$. Since $\Lambda$ is amenable, $\ker \phi$ is amenable. Note that $\phi(N_F)$, the image of $\phi$, is amenable, being a finite group.

Since $ker\phi$ and $\phi(N_F)$ are amenable groups, we conclude that $N_F$ is amenable.
\end{proof}

\begin{theorem}\label{spectralgap}
Let $\Gamma$ be a countable group and $\Lambda$ an amenable subgroup. Let $H$ and $H'$ be infinite commuting subgroups of $\Gamma$ such that $H'$ is non-amenable. Denote $\tilde M=(\tilde P\bar\otimes N)\rtimes H$ and $M=(P\bar\otimes N)\rtimes H$.\\
Let $w:H'\to \mathcal{U}(P\bar{\otimes}N)$ be a cocycle for $\rho$ and define the representation $\pi: H'\to \mathcal{U}(L^{2}(\tilde M)\ominus L^{2}(M))$ by $\pi_{g}(x)=w_{g}\tilde\rho_{g}(x)w_{g}^{*}$. Then $\pi$ has spectral gap.
\end{theorem}

\begin{remark}
In Theorem \ref{spectralgap} the action $\tilde\rho_{|H'}$ is considered to be extended in a natural way to $(\tilde P\bar\otimes N)\rtimes H$. This is possible since $H$ and $H'$ commute.
\end{remark}

\begin{proof}[Proof of Theorem \ref{spectralgap}]

Let $\mathcal{B}=\{1=\eta_{0}, \eta_{1}, ..\}\subset A$ be an orthonormal basis of $L^{2}(A).$	Denote by $u$ the canonical Haar unitary of $L(\mathbb{Z}).$ Thus, we obtain an orthonormal basis for $L^{2}(A* L(\mathbb{Z}))$ given by
$$\mathcal{\tilde{B}}=\{u^{n_{1}}\eta_{j_{1}}u^{\eta_{2}}...\eta_{j_{k}}|j_{1},..j_{k-1}\ge 1, k\in\mathbb{N}\}=\{1=\tilde{\eta_{0}}, \tilde{\eta_{1}},..\},$$
as in \cite[Proposition 2.3]{Io06a}.
Also, we have that 
\[
\mathcal{N}=\{\otimes_{f\in \Gamma/\Lambda} \eta_{i_{f}}| \{f| i_{f}\neq 0\} \mbox{ is finite}\}
\]
and
\[
\mathcal{\tilde{N}}=\{\otimes_{f\in \Gamma/\Lambda} \tilde{\eta_{i_{f}}}| \{f| i_{f}\neq 0\} \mbox{ is finite}\}
\]
are orthonormal bases for $L^2(P)$ and, respectively, for $L^2(\tilde{P})$.

Let $x=\otimes_{f\in \Gamma/\Lambda} \tilde{\eta}_{i_{f}}\in \mathcal{\tilde{N}.}$ Denote $F_{x}=\{f\in \Gamma/\Lambda|\tilde{\eta}_{i_{f}}\in\tilde{\mathcal{B}}\setminus \mathcal{B}\}$ and $K^{0}_{F}=\overline{sp}\{x\in\mathcal{\tilde{N}}|F_{x}=F\}$. Notice that $K^0_F \perp K^0_{F'}$, whenever $F\neq F'$ are finite subsets of $\Gamma/\Lambda.$ This implies that 
$$L^{2}(\tilde{P})\ominus L^{2}(P)=\overline{sp}\; \mathcal{\tilde{N}\setminus \mathcal{N}}=\oplus K^{0}_{F},$$
where the direct sum runs over all finite non empty subsets $F\subset \Gamma/\Lambda$. 

Thus,
$$L^{2}(\tilde{P}\bar\otimes N)\ominus L^{2}(P\bar\otimes N)=\oplus K^1_{F},$$
where the direct sum runs over all finite non empty subsets $F\subset \Gamma/\Lambda$ and $K^1_{F}=K_{F}^{0}\otimes L^{2}(N)$.

Finally, we get the decomposition
$$L^{2}(\tilde M)\ominus L^{2}(M)=\oplus K_{F},$$
where the direct sum runs over all finite non empty subsets $F\subset \Gamma/\Lambda$ and $K_{F}=\overline{sp}\{K^1_Fu_{h}|
h\in H\}.$

{\bf Claim 1.} We can decompose $L^{2}(\tilde M)\ominus L^{2}(M)= \oplus_{i\in I} \overline{sp}\;\pi(H')M\xi_{i} M,$ where $\{\xi_{i}\}_{i\in I}$ is a family of vectors from $L^2(\tilde P)$ and each $\xi_{i}\in K_{F}$ for some non empty finite set $F\subset \Gamma/\Lambda$. 

{\it Proof of the claim 1.}
Let $S$ be the set of elementary tensors $\otimes_{i\in F} \eta_i$, with $F$ finite subset of $\Gamma/\Lambda$ such that each $\eta_i$ is an element of $\tilde A$ which starts and ends with a non-trivial power of $u$. Then $\Gamma$ acts on $S$ and choose $T$ to be a set of representatives for this action.\\
Then, $L^{2}(\tilde M)\ominus L^{2}(M)= \oplus_{\xi\in T} \overline{sp}\;\pi(H')M\xi M.$ \hfill$\square$







Denote by $\lambda_{H'}$ the left regular representation of $H'$ on $l^2 (H').$

{\bf Claim 2.} 
$\pi\preceq \lambda_{H'}$, i.e. $\pi$ is weakly contained in $\lambda_{H'}.$

We suppose the claim holds and we prove it after the end of this theorem.
For finishing the proof, note that the non-amenability of $H'$ implies $1_{H'}\npreceq \lambda_{H'}$. Thus, $1_{H'}\npreceq \pi$, which means that $\pi$ has spectral gap on $H$. This proves the theorem.
\end{proof}

We now prove Claim 2 from the proof of Theorem \ref{spectralgap} using the same notations. 
\begin{lemma}
$\pi\preceq\lambda_{H'}.$
\end{lemma}

\begin{proof}
For every $F\subset\Gamma/\Lambda$, non empty finite subset, denote $H'_{F}=\{h'\in H'|h'F=F\}$, where $\Gamma$ acts on $\Gamma/\Lambda$ by left  multiplication. Since $F$ is finite and $\Lambda$ is amenable, Lemma \ref{amenable} implies that $H'_{F}$ is an amenable group.

Let us take a family of vectors $\{\xi_i\}_{i\in I}$ as in the first claim of Theorem \ref{spectralgap}. Fix $i\in I$ and let $F$ be a finite subset of $\Gamma/\Lambda$ such that $\xi_{i}\in K_{F}$.


Note that 
\begin{equation}\label{orthogonal}
\langle\pi_{g}(x),x\rangle=0, \forall x\in K_{F}, g\notin H'_{F}.
\end{equation}

Let us observe that we can decompose $\overline{sp}\;\pi(H')M\xi_{i} M=\oplus_{j\in J} \; \overline{sp}\; \pi (H')\eta_{ij}$ in cyclic subspaces, with $\eta_{ij}\in K_{F}$. Indeed, by taking a maximal family of vectors $\{\eta_{ij}\}_{j\in J}$ with the property that $\overline{sp}\; \pi(H'_F)\eta_{ij}$ are mutually orthogonal, we get the decomposition $\overline{sp}\;\pi(H'_{F})M\xi_{i}M=\oplus_{j\in J} \; \overline{sp}\;\pi (H'_{F})\eta_{ij}$. Since $\xi_{i}\in K_F$ and $K_{F}$ is a $\pi(H'_{F})$ invariant subspace, we obtain that $\eta_{ij}\in K_{F}$. Since $H'_F$ is a subgroup of $H'$, the decomposition $\overline{sp}\;\pi(H')M\xi_{i}M=\oplus_{j\in J} \; \overline{sp}\;\pi (H')\eta_{ij}$ also holds. Indeed, \eqref{orthogonal} implies that $\overline{sp}\;\pi(H')\eta_{ij}$ is orthogonal on $\overline{sp}\;\pi(H')\eta_{ij'}$, for all $ j,j' \in J,$ with $j\neq j'.$ This proves the claim.

Fix $j\in J.$ Define the cyclic representations $\theta:H'\to \mathcal{U}(\overline{sp}\;\pi (H')\eta_{ij})$ and $\theta_F:H'_F\to \mathcal{U}(\overline{sp}\;\pi (H'_F)\eta_{ij})$ as the restrictions of $\pi$, respectively of $\pi_{|H'_F}$ to the coresponding cyclic subspaces.

Let
 $$\tilde\theta:=\mbox{Ind}_{H'_{F}}^{H'}\theta_F:H'\to \mathcal{U}(l^{2}(H'/H'_{F})\otimes\overline {sp}\; \pi(H'_{F})\eta_{ij}) $$ be the induced representation of $\theta_F$ defined by
\begin{equation}\label{induce}
\tilde\theta_g(\delta_x\otimes\eta)=\delta_{gx}\otimes[\theta_F(c(g,x))\eta],
\end{equation}
for all $g\in H', x\in H'/H'_F$ and $\eta\in \overline {sp}\; \pi(H'_{F})\eta_{ij}$, where $c:H'\times H'/H'_{F}\to H'_{F}$ is the canonical cocycle defined as in section \ref{results}. Recall that $c(g,x)=\phi^{-1}(gx)g\phi(x),$ for all $g\in H'$ and $x\in H'/H'_F$, with $\phi:H'/H'_F\to H'$ an arbitrary fixed section. Moreover, $\phi$ can be chosen such that $\phi(H'_F)=e$, with $e$ the neutral element of $H'.$ This implies $c(g,H'_F)=g$, for all $g\in H'_F.$

{\bf Claim.}
The induced representation $\tilde\theta$ contains $\theta$ as a subrepresentation. 

{\it Proof of the Claim.}
Define the positive definite function $\varphi:H'\to \mathbb{C}$ by $\varphi(g)=<\theta(g)\eta_{ij},\eta_{ij}>$ for $g\in H'$. 
The formula \eqref{orthogonal} implies that $\varphi$ is zero on $H'\setminus H'_{F},$ since $\eta_{ij}\in K_{F}.$\\
Denote by $\tilde \eta:=\delta_{eH'_{F}}\otimes \eta_{ij}\in l^{2}(H'/H'_{F})\otimes \overline {sp}\; \pi(H'_{F})\eta_{ij}$. A direct computation gives us that 
$$<\tilde \theta (g)\tilde\eta, \tilde\eta>=<\theta(c(g,eH'_{F})\eta_{ij},\eta_{ij})>=<\theta(g)\eta_{ij},\eta_{ij})>=\varphi(g)$$
for all $g\in H'_{F}.$
For $g\notin H'_F$, the formula \eqref{induce} gives us that$<\tilde \theta (g)\tilde\eta, \tilde\eta>=0.$\\
Finally, we have obtained that
$$<\tilde \theta (g)\tilde\eta, \tilde\eta>=\varphi(g)$$
for all $g\in H'$. Since $\theta$ is a cyclic representation, we get that $\theta$ is contained in $\tilde \theta.$ This ends the claim. \hfill$\square$

Now, we can finish the proof of the lemma. Since $H'_{F}$ is amenable, we have $1_{H'_{F}}\preceq \lambda_{H'_{F}}$ (see \cite[Theorem G.3.2]{BHV08}, for example). By Fell absorbing principle, we get that $\theta_F\preceq \lambda_{H'_{F}}$. \cite[Theorem F.3.5]{BHV08}  gives us continuity of weak containment with respect to induction. This implies that $\tilde\theta=\mbox{Ind}_{H'_{F}}^{H'}\theta_F\preceq \mbox{Ind}_{H'_{F}}^{H'}\lambda_{H'_{F}}=\lambda_H'$. Since $\theta$ is contained in $\tilde\theta$, we get that $\theta\preceq \lambda_{H'}$.  

Denote by $\theta_i:H'\to \mathcal{U}(\overline{sp}\;\pi(H')M\xi_{i} M)$ the restriction of $\pi$ to the subspace $ \overline{sp}\;\pi(H')M\xi_{i} M.$
The decomposition $\overline{sp}\;\pi(H')M\xi_{i} M=\oplus \overline{sp}\;\pi(H')\eta_{ij}$ gives us that $\theta_i\preceq\lambda_{H'}$.


The decomposition given by the first claim implies that $\pi=\oplus_{i\in I}\theta_i$. Thus, $\pi\preceq \lambda_{H'},$ which ends the proof of the lemma.
\end{proof}

{\bf Proof of Theorem \ref{thmB}}


Define the representation $\pi: \Gamma\to \mathcal{U}(L^{2}((\tilde P\bar\otimes N)\rtimes H)\ominus L^{2}((P\bar\otimes N)\rtimes H))$ by $\pi_{g}(x)=w_{g}\tilde\rho_{g}(x)w_{g}^{*}$ and denote $\tilde M=(\tilde P\bar\otimes N)\rtimes H$ and $M=(P\bar\otimes N)\rtimes H$ as in the previous theorem.

Theorem \ref{spectralgap} implies that $\pi$ has spectral gap on $H'$. Thus, for all $\epsilon>0$, exists $\delta>0$ and $F'\subset H'$ finite, such that if $u\in\mathcal{U}(\tilde M)$ satisfies $\|\pi_{h}(u)-u\|_{2}\leq \delta, \forall h'\in F'$, then $\|u-E_{M}(u)\|_{2}\leq \epsilon$.

Let us proceed now as in [\cite{Po06},Theorem 4.1] . Denote by $\bar u_{g}:=w_{g}u_{g}, g\in\Gamma$. Since the map $s\to\theta_{s}(\bar u_{s})$ is continuous in $\|\cdot\|_2$ for all $h'\in F'$, it follows that for small enough $s$, we get that
$$\|\theta_{-s/2}(\bar u_{h})-\bar u_{h}\|_{2}\leq \delta/2,$$
for all $h'\in F'.$ Because $H$ and $H'$ commute, $\bar u_{h'}$ and $\bar u_{g}$ commute for all $h'\in F'$ and $g\in H$. Thus, we get that 
$$\|[\theta_{s/2}(\bar u_{g}),\bar u_{h'}]\|_{2}=\|\bar u_{g},\theta_{-s/2}(\bar u_{h'})\|_{2}\leq 2\|\theta_{-s/2}(\bar u_{h'})-\bar u_{h'}\|\leq \delta,$$
for all $h'\in F'$ and $g\in H.$

Notice that $\pi_h'(x)=\bar u_h' x \bar u_h'^*$, for all $g\in H'.$ A direct computation gives us that
$$\|\pi_{h'}(\theta_{s/2}(\bar u_{g}))-\theta_{s/2}(\bar u_{g})\|_{2}=\|[\theta_{s/2}(\bar u_{g}),\bar u_{h'}]\|_{2}\leq \delta ,$$
for all $h'\in F'$ and $g\in H$, which implies that
$$\|\theta_{s/2}(\bar u_{g})-E_{M}(\theta_{s/2}(\bar u_{g}))\|_{2}\leq \epsilon.$$

Using Lemma \ref{transversality}, we get that 
$\|\theta_{s}(\bar u_{g})-\bar u_{g}\|_{2}\leq 2\epsilon$
for all $g\in H.$ The set $K:=\overline{co}^{w}\{\bar u_{g}\theta_{s}(\bar u_{g})^*| g\in H\}$ is convex weakly compact and for all $\xi\in K$ and $g\in H$, we have $\bar u_{g}\xi \theta_{s}(\bar u_{g})^*\in K.$ Thus, if we denote by $\xi_{0}\in K$ the unique element of minimal norm $\|\|_{2}$, then we get that $\bar u_{g}\xi_{0}\theta_{s}(\bar u_{g})^*=\xi_{0}$ for all $g\in H$. This is equivalent to
$$w_{g}\tilde\rho_{g}(\xi_{0})=\xi_{0}\theta_{s}(w_{g}),$$
for all $g\in H.$
Taking $v\in \tilde P \bar\otimes N$, to be the partial isometry of $\xi_{0}$, we get that
$$w_{g}\tilde\rho_{g}(v)=v\theta_{s}(w_{g}),$$
for all $g\in H.$

Since 
$$\|\bar u_g\theta_{s}(\bar u_{g})^*-1\|_{2}=\|\bar u_g-\theta_{s}(\bar u_{g})\|_{2}\leq 2\epsilon,$$
we get that $\|\xi_{0}-1\|_{2}\leq 2\epsilon$, which implies that $\|v-1\|_{2}\leq 4(2\epsilon)^{1/2}$. This proves Step 1 of the proof of Proposition \ref{theta}. In combination with Step 2 from the proof of Proposition \ref{theta}, the conclusion follows as in the proof of Proposition \ref{theta}. Meaning, we obtain that $w_{|H}$ and $\theta(w)_{|H}$ are cohomologous.

As in the proof of Theorem \ref{thmA}, we use Theorem \ref{final} to deduce the existence of a unitary $u\in \mathcal{U}(P\bar\otimes N)$ and of a cocycle $w':H\to\mathcal{U}(N)$ such that
$$w_{h}=uw'_{h}\rho_{h}(u^*), \quad \forall h\in H.$$

We have $H\lhd HH'$, because $H$ and $H'$ commute. Since the restriction of $\rho$ to $H$ is weakly mixing, by using Proposition \ref{normal} we obtain a cocycle $w'$ with values in $N$ which is cohomologous to $w$ on $HH'$.

For the moreover part, we apply again Proposition \ref{normal} as in the proof of Theorem \ref{thmA}.
\hfill$\square$

\section{Applications to W$^*$-superrigidity}

We record the results \cite[Corollary 5.3]{CP10}, \cite[Theorem 1.1]{PV12}, \cite[Theorem 1.1]{Io12b} in the following theorem, which give uniqueness of group measure space Cartan subalgebras for groups in $\mathcal C$.

\begin{theorem}\label{W*}
Let $\Gamma\in\mathcal C$ and let $\Gamma\curvearrowright X$ be a free ergodic pmp action on a standard probability space $X$. Suppose there exists $\Lambda\curvearrowright Y$ a free ergodic pmp action on a standard probability space $Y$ such that $M=L^\infty(X)\rtimes \Gamma=L^\infty(Y)\rtimes \Lambda.$ Then there exists a unitary $u\in M$ such that $uL^\infty(X)u^*=L^\infty(Y).$
\end{theorem}

The following result is a particular case of \cite[Theorem 5.6]{Po05} (see also \cite[Theorem 1.8]{Fu06}).


\begin{proposition}\cite[Theorem 5.6]{Po05}\label{superrigid}
Let $\Gamma$ be a countable group with no non-trivial finite normal subgroups. Let $\Gamma\curvearrowright (X,\mu)$ be a free pmp action, where $(X,\mu)$ is a standard probability space. If $\Gamma\curvearrowright (X,\mu)$ is $\mathcal U_{fin}$-cocycle superrigid, then $\Gamma\curvearrowright (X,\mu)$ is OE-superrigid.
\end{proposition}
 
The following lemma gives a sufficient condition for coinduced actions to be free.

\begin{lemma}\label{freecoinduce}\cite[Lemma 2.1]{Io06b}
Let $\Gamma$ be a countable group and $\Lambda$ a subgroup of infinite index. Let $\Lambda\overset{\sigma_0}\curvearrowright (X_0,\mu_0)$ be a pmp action on the standard probability space $(X_0,\mu_0)$ which has no atoms and let $\Gamma\overset{\sigma}\curvearrowright (X,\mu):=(X_0,\mu_0)^{\Gamma/\Lambda}$ be the coinduced action. Suppose $\cap_{g\in\Gamma}g\Lambda g^{-1}$ is finite and $\cap_{g\in\Gamma}g\Lambda g^{-1}\cap Fix(\Lambda\curvearrowright X_0)=\{e\}$, where $Fix(\Lambda\curvearrowright X_0)$ consists of those elements $g\in \Lambda$ for which $\{x_0\in X_0|gx_0=x_0\}$ has measure $1.$ Then $\Gamma\curvearrowright X$ is free.
\end{lemma}
\begin{proof}
Define $A_g=\{(x_h)_{h\in\Gamma/\Lambda}\in X|\sigma_g((x_h)_h)=(x_h)_h\}$ for $g\in\Gamma$. Recall that $\sigma_g((x_h)_h)=(x'_{h})_h$, where $x'_h=\phi^{-1}(gh)g\phi(h)x_{g^{-1}h}$ and $\phi:\Gamma/\Lambda\to\Gamma$ is a section.  \\
If $g_0\notin\cap_{g\in\Gamma}g\Lambda g^{-1},$ there exists $g_1\in\Gamma$ such that $g_0^{-1}g_1\Lambda\neq g_1\Lambda.$ Then $$
\begin{array}{rcl}
A_{g_0}&=&\{(x_h)_{h}\in X|x_h=\phi(g_0h)^{-1}g_0\phi(h)x_{g_0^{-1}h}, \forall h\in\Gamma/\Lambda\}\\
&\subset&\{(x_h)_{h}\in X|x_{g_1\Lambda}=\phi(g_0g_1\Lambda)^{-1}g_0\phi(g_1\Lambda)x_{g_0^{-1}g_1\Lambda}\}
\end{array}
$$
has measure $0$ since $X_0$ is non-atomic.

Now, if $g_0\in \Sigma:=\cap_{g\in\Gamma}g\Lambda g^{-1}\setminus \{e\}$, we have $g^{-1}g_0g\in \Sigma\setminus\{e\},$ for all $g\in\Gamma$. The hypothesis implies that $C_{\lambda}:=\{x_0 \in X_0|\lambda x_0=x_0\}$ has measure less than $1$, for all $\lambda \in \Sigma\setminus \{e\}$. Then,
$$
\begin{array}{rcl}
A_{g_0}&=& \{(x_h)_{h\in\Gamma/\Lambda}|x_h=\phi(h)^{-1}g_0\phi(h)x_h,\forall h\in\Gamma/\Lambda\}\\ 
&=& \prod_{h\in \Gamma/\Lambda} C_{\phi(h)^{-1} g_0\phi(h)}
\end{array}
$$
has measure $0$. Indeed, since $\Sigma$ is finite, there exists $g_1\in \Sigma\setminus\{e\}$ such that $\{h\in \Gamma/\Lambda| \phi(h)^{-1}g_0\phi(h)=g_1\}$ is an infinite set. This implies $A_{g_0}$ has measure $0$ since $\mu_0(C_{g_1})<1$.
\end{proof}

Notice that the proof of Lemma \ref{freecoinduce} also proves that if we coinduce from free actions, we obtain free actions. The following result proves cocycle superrigidity for coinduced actions of groups from $\mathcal C.$
 

\begin{theorem}\label{csuperrigid}
Let $\Gamma\in\mathcal C$ and $\Lambda$ a subgroup defined as in Corollary \ref{WW*}. Let $\Lambda \curvearrowright X_0$ be a measure preserving action on a standard probability space $X_0$ and let $\Gamma\curvearrowright X$ be the coinduced action from $\Lambda \curvearrowright X_0$. Then $\Gamma\curvearrowright X$ is $\mathcal U_{fin}$-cocycle superrigid.
\end{theorem}
\begin{proof}
We apply Theorems \ref{A} and \ref{B} and let us use the notations from these theorems.

For $\Gamma\in \mathcal C_1$, we want to apply Theorem \ref{B}. If we take $H'=\Gamma_1$ and $H=\Gamma_2$, the conditions of Theorem \ref{B} are satisfied, so we obtain the claim.

If $\Gamma\in \mathcal C_2$, consider $\Gamma=\Gamma_1\times\Gamma_2\times...\times\Gamma_n$, with all the $\Gamma_i$'s non-elementary hyperbolic groups. Without loss of generality suppose that $\Lambda\subset\Gamma_1$. We apply again Theorem \ref{B}. By taking $H'=\Gamma_1$ and $H=\Gamma_2\times...\times\Gamma_n$ we notice that the conditions of this theorem are again satisfied.

Let $\Gamma\in \mathcal C_3$. Since $\Sigma^0$ is contained in $\Gamma^0_1$, the hypothesis implies that $\Gamma^0_1$ does not have finite index subgroups which are contained in a conjugate of $\Lambda$. We want to apply Theorem \ref{A} for $H=\Gamma^0_1.$\\
Take $\mathcal V\in\mathcal U_{fin}$ and $w:\Gamma\times X\to \mathcal V$ a cocycle for $\Gamma\curvearrowright X.$ Theorem \ref{A} implies that there exist $\phi:X\to\mathcal V$ such that $\phi(gx)^{-1}w(g,x)\phi(x)$ is independent of $x$ on $\Gamma^0_1.$ Since $\Sigma^0$ is normal in $\Gamma^0_2$ and contained in $\Gamma^0_1$, Lemma \ref{weakmixing} combined with Proposition \ref{normal}  implies that $\phi(gx)^{-1}w(g,x)\phi(x)$ is independent of $x$ on $\Gamma^0_2.$ This proves that $w$ is cohomologous with a group homomorphism on $\Gamma^0_1*_{\Sigma^0}\Gamma^0_2$. \\ 
Now, since  $\Gamma^0_1*_{\Sigma^0}\Gamma^0_2$ is normal in $\Gamma$, we apply again Lemma \ref{weakmixing} and Proposition \ref{normal} to obtain that $\phi(gx)^{-1}w(g,x)\phi(x)$ is independent of $x$ on $\Gamma.$ This ends the proof.
\end{proof}

{\it Proof of the Corollary \ref{WW*}.}
Combining Proposition \ref{superrigid} with Theorem \ref{csuperrigid}, we obtain that $\Gamma\curvearrowright X$ is OE-superrigid. 
Lemma \ref{weakmixing} proves that $\Gamma\curvearrowright X$ is ergodic and we conclude using Theorem \ref{W*}.
\hfill$\square$

\end{document}